\documentclass{amsart}

\usepackage{amsmath}
\usepackage{amsfonts}
\usepackage{amssymb}
\usepackage{amsthm}
\usepackage{bm}
\usepackage{enumerate}
\usepackage{algorithm}
\usepackage{algorithmicx}
\usepackage{algpseudocode}

\allowdisplaybreaks

\newtheorem{theorem}{Theorem}[section]
\newtheorem{lemma}[theorem]{Lemma}

\newcommand{\R}{\mathbb{R}}

\newcommand{\Z}{\mathbb{Z}}
\newcommand{\N}{\mathbb{N}}

\newcommand{\B}{{\mathcal{B}}}

\begin{document}
\title{On strictly nonzero integer-valued charges}
\author{Swastik Kopparty}
\thanks{Swastik Kopparty is supported in part by a Sloan
Fellowship and NSF grants CCF-1253886 and CCF-1540634.}
\address{Department of Mathematics \& Department of Computer Science,
Rutgers University, Piscataway NJ 08854}
\email{swastik@math.rutgers.edu}
\author{K.P.S. Bhaskara Rao}
\address{Department of Computer Information Systems,
Indiana University Northwest, Gary IN 46408}
\email{bkoppart@iun.edu}

\begin{abstract}

A charge (finitely additive measure) defined on a Boolean algebra of sets taking values in a group $G$ is called a strictly nonzero (SNZ) charge if it takes the identity value in $G$ only for the zero element of the Boolean algebra. A study of such charges was initiated by Rudiger G{\"o}bel and KPS Bhaskara Rao in 2002~\cite{GB}. 

Our main result is a solution to one of the questions posed in that paper: we show that for every cardinal $\aleph$, the Boolean algebra of clopen sets of $\{0,1\}^\aleph$ has a strictly nonzero integer-valued charge. The key lemma that we prove is that there exists a strictly nonzero integer-valued {\em permutation-invariant} charge on the Boolean algebra of clopen sets of $\{0,1\}^{\aleph_0}$.
Our proof is based on linear-algebraic arguments, as well as certain kinds of polynomial approximations of binomial coefficients.

We also show that there is no integer-valued SNZ charge on ${\mathcal{P}}(N)$. Finally, we raise some interesting problems on integer-valued SNZ charges.
\end{abstract}

\maketitle

\section{ Introduction}

If $G$ is a group and $\mathcal A$ is a Boolean algebra, when does there exist a strictly nonzero $G$-valued charge (finitely additive measure) on $\mathcal A$? This problem was posed by G{\"o}bel and Bhaskara Rao in {\cite{GB}}, and several results about this general question were proved there. 

Even the special cases of the above problem when the group $G$ equals the group of real numbers $\mathbb R$, the group of rational numbers $\mathbb Q$, or the group of integers $\Z$, are all interesting and suggest many challenging problems in the intersection of combinatorics, group theory, and set theory.

Kelley~\cite{JLK} gave necessary and sufficient conditions for the existence of a bounded strictly positive $\R$-valued charge. As was observed in~\cite{B}, this also provides a necessary and sufficient condition for the existence of a bounded strictly nonzero $\R$-valued charge.

Regarding the existence of $\Z$-valued SNZ charges, some necessary conditions were derived in \cite{GB}. For example, it was shown that if a Boolean algebra $B$ is nonatomic and if there is a $\Z$-valued SNZ charge on $B$, then $B$ should satisfy the countable chain condition (every collection of pairwise disjoint nonzero elements of $B$ is countable). It was also shown that if there is a $\Z$-valued SNZ charge on a Boolean algebra $B$, then every chain of distinct elements in $B$ is countable. 

In \cite{GB}, the question was raised as to whether the above two necessary conditions guarantee the existence of a $\Z$-valued SNZ charge.

The  Boolean algebra of clopen sets of $\{0,1\}^{\aleph}$ for an infinite cardinal $\aleph$ (denoted $\B(2^{\aleph})$) is a nonatomic Boolean algebra and satisfies both the above necessary conditions, namely, the
countable chain condition and the condition that every chain is countable.
In this context the question was raised as to whether this Boolean algebra admits a $\Z$-valued SNZ charge.

 Our main result is that $\B(2^{\aleph})$ has an SNZ $\Z$-valued charge.

\begin{theorem}
\label{thm1}
For every infinite cardinal $\aleph$, $B(2^{\aleph})$ has a strictly nonzero $\Z$-valued charge.
\end{theorem}

The above theorem for the case of $\aleph = \aleph_0$ follows from Proposition 13 of \cite{GB}, which showed that every countable Boolean algebra has an SNZ $\Z$-valued charge. In~\cite{GB}, it was suggested that the answer to this question
might depend on the axioms of set theory (in particular, on large cardinal axioms). Our results show that they do not.


The key ingredient of our proof of Theorem~\ref{thm1} is the existence of a {\em permutation-invariant} $\Z$-valued charge on
$\B(2^{\aleph_0})$. Propositions 12 and 15 of \cite{GB} together\footnote{The terminology of~\cite{GB} is different from ours. In~\cite{GB}, a strictly nonzero charge on $\B(2^\aleph)$ is 
referred to as a ``good" charge, and a permutation-invariant strictly nonzero charge on $\B(2^{\aleph})$ is referred to as, of course,  a ``very good" charge.} show that the existence of such a charge on $\B(2^{\aleph_0})$
implies the existence of a strictly nonzero $\Z$-valued charge on $\B(2^{\aleph})$ for every uncountable cardinal $\aleph$. Thus the following theorem
implies Theorem~\ref{thm1}.

\begin{theorem}
\label{thm2}
$B(2^{\aleph_0})$ has a permutation-invariant strictly nonzero $\Z$-valued charge.
\end{theorem}

We prove Theorem~\ref{thm2} in Section~\ref{sectionmain}.
At its core, Theorem~\ref{thm2} is a statement about the existence of integer solutions to a 
certain countable system of linear inequations in countably many variables. 
The coefficients of these linear inequations are related to binomial
coefficients. We use linear algebraic arguments, as well as 
some polynomial approximations to binomial coefficients, to show
the existence of an integer solution to the given system of
inequations.

In Section~\ref{sectionPN}, we show that there is no SNZ charge on ${\mathcal P}(N)$. We conclude with some open problems.

\section{Notation and Preliminaries}

All $\log$s are to the base $2$.	
We define ${ 0 \choose 0} = 1$, and if $ b < 0 $ or $b > a$, then ${a \choose b } = 0$.

We recall some notation from \cite{GB}.

If $A$ and $B$ are finite disjoint subsets of an  index set $Y$ of cardinality $\aleph$, let $H(A, B) = \{ f \in \{0,1\}^Y: f(y) = 0   \mbox{  for  } y \in A \mbox{  and  } f(y) = 1  \mbox{  for  }   y \in B\}$.
Recall that a subset of $\{0,1\}^{Y}$ is {\em clopen} if
it can be expressed as the union of finitely many sets of the form
$H(A,B)$ with $A$, $B$ both finite.

Let $Y$ be an index set with cardinality $\aleph$.
Let $\mu$ be a $\Z$-valued charge on $\B(2^{Y})$.
We say that $\mu$ is {\em permutation-invariant} if
for all permutations $\pi: Y \to Y$ and all clopen sets $U$,
we have $\mu(\pi(U)) = \mu(U)$ (where for a set $U \subseteq \{0,1\}^Y$,
$\pi(U)$ is defined to equal $\{ f \circ {\pi^{-1}} \mid f \in U \}$).

It is easy to see that $\mu$ is permutation invariant if
and only if $\mu(H(A, B))$ depends only on the cardinalities of $A$ and $B$.

Let $\mu$ is a permutation invariant $\Z$-valued charge on $\B(2^Y)$.
Define
$h: \mathbb N \times \mathbb N \to \mathbb \Z$
by:
$$ h(m, n) = H(A, B),$$
for any disjoint $A, B$ with $|A| = k$, $|B| = k'$.
By finite additivity, we have 
$$h(m, n) = h(m+1, n) + h(m, n+1).$$
Using this relation, and letting $p_n = h(n, 0)$, it 
follows by induction
that the $p_n$ determine the $h(m,n)$ via the following
simple formula:
\begin{align}
\label{hpformula}
h(m,n) = \sum_{i = 0}^n (-1)^i {n \choose i} p_{m+i}.
\end{align}

Conversely, given any sequence of integers $p_0, p_1, \ldots$, 
if we define $h(m,n)$ by the above formula,
we get a $\Z$-valued charge $\mu$ defined by:
\begin{align}
\label{muhformula}
\mu(A,B) = h(|A|, |B|).
\end{align}

We now express the condition of strict nonzeroness
of a permutation invariant measure in terms of
the $h(m,n)$.
For every clopen set $U$ in $\B(2^Y)$,
there is a finite set $C \subseteq Y$ of size $t$ such that
$U$ can be expressed as the disjoint union of sets
of the form $H(A, B)$, with $A \cup B = C$ and $A \cap B = \emptyset$.
Thus $\mu(U)$ is of the form:
$$ \sum_{j =0 }^t w_j h(j,t-j),$$
where $w_j$ is an integer with $0 \leq w_j \leq {t \choose j}$ (here $w_j$ represents the number of  $A, B$ pairs appearing in the above representation of $U$  with $|A| = j$).

We thus get the following criterion for strict nonzeroness of a charge.
Suppose we define a permutation-invariant $\Z$-valued charge $\mu$ on $\B(2^{Y})$ by specifying integers $p_0, p_1, \ldots$, and then defining
$h$ and $\mu$ by \eqref{hpformula} and \eqref{muhformula}.
Then $\mu$ is strictly nonzero if for all integers $t \geq 0$,
and for integers $w_0, w_1, \ldots, w_t$, not all zero, with $0 \leq w_j \leq { t \choose j }$,
$$ \sum_{j = 0}^t w_j h(j, t-j) \neq 0.$$

\section{ A permutation-invariant $\Z$-valued SNZ charge on $B(2^{Y})$}
\label{sectionmain}
The following theorem shows that if we pick the integers
$p_0, p_1, \ldots, $ growing sufficiently rapidly, then
the permutation-invariant $\Z$-valued charge defined on $\B(2^{\aleph})$
through the process described in the previous section
is strictly nonzero. This implies both Theorem~\ref{thm1}
and Theorem~\ref{thm2}.

\begin{theorem}
\label{mainthm}	
Define $f(k) = 2^{(100k)^{10}}$.

Let $p_0, p_1, \ldots$ be a sequence of integers such that $p_0 \neq 0$,
and for each $k\geq 1$,
$$|p_k| > f(k) \cdot (\sum_{i=0}^{k-1} |p_i|).$$

Define $h(m,n) = \sum_{i = 0}^n (-1)^i {n \choose i} p_{m+i}.$

Then for every $t \geq 0$,
and for integers $w_0, w_1, \ldots, w_t$, not all zero, with $0 \leq w_j \leq { t \choose j }$, we have:
$$ \sum_{j = 0}^t w_j h(j, t-j)  \neq 0.$$
\end{theorem}
\begin{proof}
Suppose not. That is, there exists a $t$ and integers $w_0, w_1, \ldots, w_t$, not all zero,
with $0 \leq w_j \leq {t \choose j}$ such that 
$$ \sum_{j = 0}^t w_j h(j, t-j) = 0.$$

Expanding the $h(m,n)$ in terms of the $p_i$, we get:
$$ \sum_{j=0}^t w_j \sum_{i=0}^{t-j} (-1)^{i} {t-j \choose i} p_{j + i} = 0,$$
which, after re-indexing in terms of $k = j+i$ and simplifying, gives us:
$$ \sum_{j=0}^t \sum_{k=0}^t   (-1)^{k - j} {t-j \choose t-k} w_j p_k  = 0,$$
(here we used the  fact that ${t-j \choose k-j } = {t-j \choose t - k}$).

Let $M$ be the matrix with rows and columns indexed by $\{0, 1, \ldots, t\}$,
whose $(j,k)$ entry is given by:
$$M_{j,k} = (-1)^{k-j} {t-j  \choose t-k}.$$

Let $v_k \in \mathbb Z^{t+1}$ denote the $k^{\rm{th}}$ column of this matrix.
Let $w$ denote the vector $(w_0, w_1, \ldots, w_t) \in \mathbb Z^{t+1}$.

In this notation, we have:
$$ \sum_{k=0}^{t} \langle w, v_k \rangle p_k = 0.$$

Observe that the $v_k$ form a basis for $\mathbb R^{t+1}$ (since the $v_k$ are
``upper triangular"). 
By assumption, $w$ is not the $0$ vector, and so there exists some $k$
such that $\langle w, v_k \rangle \neq 0$.
Let $s$ be the largest such $k$.
Then:
$$ \sum_{k=0}^{s} \langle w, v_k \rangle p_k = 0.$$

Observe that if $s = 0$, then we immediately have a contradiction to the above equation. Thus we may assume that $s \geq 1$.

\begin{lemma}
\label{lemorthog}
	$s \leq \frac{1}{100} (\log t)^{1/5}$.
\end{lemma}
\begin{proof}
Suppose $s > \frac{1}{100} (\log t)^{1/5}$.

By the formula above, we have:
$$p_{s} = \frac{-1}{\langle w, v_{s} \rangle} \sum_{k=0}^{s - 1} \langle w, v_k \rangle p_k.$$

Using the bounds we know on the coordinates of $w$ and the $v_k$, 
we have $$|\langle w, v_k \rangle | \leq \sum_{j = 0}^{t} {t \choose j} \cdot {t-j \choose t-k} \leq (k+1) \cdot t^k$$ for each $k$.
Also $|\langle w, v_{s} \rangle| \geq 1$, by integrality.
Thus: $$|p_{s}| \leq  s \cdot t^{s-1} \cdot (\sum_{k=0}^{s-1} |p_k| ).$$
Now if $s > \frac{1}{100} (\log t)^{1/5}$, then $s \cdot t^{s-1} < f(s)$ (since
$s \cdot t^{s-1} \leq t^{s}$, and $f(s)^{1/s} \geq 2^{(100 s)^{9}} > t$).
Thus:
$$|p_{s}| \leq f(s) \cdot (\sum_{k=0}^{s-1} |p_k| ).$$
This contradicts the hypothesis: 
$$ |p_{s}| > f(s) \cdot (\sum_{k=0}^{s-1} |p_k| ).$$

Thus $s \leq \frac{1}{100} (\log(t))^{1/5}$.
\end{proof}

For $i \in \{0,1, \ldots, t\}$, let $u_i \in \Z^{t+1}$
be the vector given by:
$$u_i = ( {t-i \choose t }, {t-i \choose t-1}, \ldots, {t-i \choose 0} ).$$
Note that the first $i$ coordinates of this vector are $0$.

The next lemma shows that the $u_i$ vectors are a dual basis to the $v_i$ vectors. This fact is very old and classical, and we include a quick proof 
in the appendix for completeness.
\begin{lemma}
\label{lemdual}
$$ \langle u_i, v_k \rangle = \begin{cases} 1 & i = k \\ 0 & i \neq k \end{cases}.$$
\end{lemma}


By Lemma \ref{lemdual}, we know that $w$ is in the span of $u_0, u_1, \ldots, u_{s}$, 
and that $w$ is not in the span of $u_0, \ldots, u_{s-1}$.

Let $b_0, \ldots, b_s \in \mathbb R$ be such that $w = \sum_{i=0}^s b_i u_i$. 
Let $b$ be the row vector $(b_0, b_1, \ldots, b_s)$.
Observe that the $i$ coordinate of $u_i$ equals $1$, and for $j < i$, the $j$ coordinate of $u_i$
equals $0$. Thus, the $u_i$ are ``upper triangular", and since
$w \in \Z^{t+1}$, we get that $b_0, b_1, \ldots, b_s$ are all in $\Z$.
Furthermore, $b_s \neq 0$.

Using the equation $\sum_{k=0}^s\langle w, v_k \rangle p_k = 0$ along with Lemma~\ref{lemdual},
we get:
\begin{align}
\label{eqbp}
\sum_{k=0}^s b_k p_k = 0.
\end{align}

We will now show that the three facts:
\begin{itemize}
	\item $w = \sum_{i=0}^s b_i u_i$,
	\item $s \leq \frac{1}{100} (\log t)^{1/5}$,
	\item $0 \leq w_j \leq {t \choose j}$ for each $j$,
\end{itemize}
together imply that the $b_i$ are small, in the sense that $\sum_{i=0}^s |b_i| \leq (20s)^{20s^2}$.
This, combined with the fact that $b_s \neq 0$ and the equality
$\sum_{k=0}^s b_k p_k =0$, will contradict the rapid growth of the $p_k$.

Let $P$ be the $(s+1) \times (t+1)$ matrix whose rows are $u_0, u_1, \ldots, u_s$.
Then $ b \cdot P = w$. 

We know that $0 \leq w_j \leq {t \choose j}$. We now use this to deduce some information about the vector $b$.

Define $\tilde{P}$ to be the $(s+1) \times (t+1)$ matrix which is obtained from $P$ as follows:
for each $j \in \{0, 1, \ldots, t\}$, divide column $j$ of $P$ by ${t \choose j}$.
Thus $b \cdot \tilde{P}$ is a vector with all its coordinates lying in $[0,1]$.

Let us study the matrix $\tilde{P}$.
The $i,j$ entry of $\tilde{P}$ is given by:
\begin{align}
\tilde{P}_{i,j} &= \frac{ {t-i \choose t-j} }{ {t \choose j }}\\
&= \frac{ {t-i \choose t-j} }{ {t \choose t-j }}\\
&= \frac{ (t-i) (t-i-1) \ldots (j-i+1) }{ { t (t-1) \ldots (j+1)}}.
\end{align}
If $i < j$ and $i < t-j$, then we can cancel many common terms, and we get:
$$\tilde{P}_{i,j} = \frac{ j (j-1) \ldots (j-i+1) }{ { t (t-1) \ldots (t-i+1)}}.$$
Thus we have:
\begin{align}
	\label{binomapprox}
	(\frac{j - i + 1}{t})^i \leq \tilde{P}_{i,j} \leq (\frac{j}{t-i + 1})^i
\end{align}

The rest of the argument is motivated by the following observation.
If $t$ is very large relative to $s$ (as we know it is),
then the above expression implies that $\tilde{P}_{i,j}$ is approximately $(\frac{j}{t})^i$.
Thus $\tilde{P}$ is approximately a Vandermonde matrix. 
This will enable us to express what we know about $b \cdot \tilde{P}$ 
in terms of evaluations of the polynomial $R(X) = \sum_{i=0}^s b_i X^i$.

\begin{lemma}
\label{bbound}
$\sum_{i=0}^s |b_i| \leq (20s)^{20s^2}$.
\end{lemma}
\begin{proof}
Let $C = \sum_{i=0}^s |b_i|$.

For $\ell \in \{1, 2, \ldots, s+1\}$, 
define $\lambda_\ell \in \{0,1, \ldots, t\}$ by:
$$\lambda_\ell = \lfloor\left(\frac{\ell}{s+2}\right) \cdot t\rfloor,$$
and let $y_\ell \in \Z^{t+1}$ be the $(\lambda_\ell)^{\rm{th}}$ column of $\tilde{P}$. We thus have $\langle b, y_\ell \rangle \in [0,1]$ for all $\ell \in \{1, \ldots, s+1\}$.

Define the polynomial $R(X) = \sum_{i=0}^{s} b_i X^i$.

The strategy is in two steps.
We will first show that for each $\ell \in \{1,2,\ldots, s+1\}$,
$$R(\frac{\ell}{s+2}) \approx \langle b, y_{\ell} \rangle.$$
We will then show that if $C$ is large, then $R(\frac{\ell}{s+2})$ must be $\gg 1$ for
some $\ell$. This will contradict the fact that
$\langle b, y_{\ell} \rangle \leq 1$.

\begin{lemma}
\label{lemerr}
	For each $\ell \in \{1,2, \ldots, s+1\}$,
	\begin{align}
\label{lemeqerr}
|\langle b, y_{\ell} \rangle - R(\frac{\ell}{s+2}) | \leq \frac{1}{t^{1/4}} \cdot C.
\end{align}
\end{lemma}
\begin{proof}
We have:
\begin{align}
\langle b, y_{\ell}\rangle - R(\frac{\ell}{s+2}) &= \sum_{i=0}^s b_i \tilde{P}_{i,\lambda_\ell} - \sum_{i=0}^s b_i \left(\frac{\ell}{s+2}\right)^{i}\\
&= \sum_{i=0}^s b_i  \left( \tilde{P}_{i,\lambda_\ell} - \left(\frac{\ell}{s+2}\right)^i \right) \\
&\leq \left(\sum_{i=0}^{s} |b_i| \right) \cdot \max_{i} \left| \tilde{P}_{i,\lambda_\ell} - \left(\frac{\ell}{s+2}\right)^i  \right|.
\label{eqerr}
\end{align}

We now estimate 
$$\left| \tilde{P}_{i,\lambda_\ell} - \left(\frac{\ell}{s+2}\right)^i  \right|.$$
Since $i \leq s < \frac{t}{s+2} - 1 \leq \lambda_\ell$ and 
$i \leq s < \frac{t}{s+2}-1 \leq t - \lambda_\ell$,
we may use equation $\eqref{binomapprox}$ to bound $\tilde{P}_{i, \lambda_\ell}$.
We thus get the upper bound:
\begin{align*}
\tilde{P}_{i, \lambda_\ell} &\leq\left( \frac{\lambda_\ell}{t - i + 1} \right)^i\\
&\leq \left( \frac{\frac{\ell}{s+2} \cdot  t + 1}{t - i + 1} \right)^i\\
&\leq \left( \frac{\frac{\ell}{s+2} \cdot t + 1 }{t - s + 1} \right)^i\\
&\leq \left( \frac{\ell}{s+2}  + \frac{s}{t-s+1}\right)^i \\
&\leq \left( \frac{\ell}{s+2} \right)^{i} \left( 1 + \frac{s (s+2)}{\ell(t-s+1)} \right)^i\\
&\leq \left( \frac{\ell}{s+2} \right)^{i} \left( 1 + \frac{s (s+2)}{(t-s+1)} \right)^s\\
&\leq \left( \frac{\ell}{s+2}\right)^i  e^{4s^3/t},
\end{align*}
where in the last step we used the elementary inequality $(1 + x) \leq e^{x}$ for all $x$.
Similarly, we get the lower bound:
\begin{align*}
\tilde{P}_{i, \lambda_\ell} & \geq \left( \frac{\frac{\ell}{s+2} t - i }{t} \right)^i\\
&\geq \left( \frac{\ell}{s+2} - \frac{i}{t} \right)^i\\
&\geq \left( \frac{\ell}{s+2} \right)^i \left( 1  - \frac{i (s+2)}{\ell \cdot t} \right)^i\\
&\geq \left( \frac{\ell}{s+2} \right)^i \left( 1  - \frac{s(s+2)}{t} \right)^s\\
&\geq \left( \frac{\ell}{s+2} \right)^i e^{-4s^3/t},
\end{align*}
where in the last step we used the elementary inequality $1-x \geq e^{-2x}$ for all $x \in [0,\frac12]$.
Now since $s \leq \frac{1}{100} (\log t)^{1/5} < \frac{1}{100} t^{1/4}$, we have $4s^3/t < \frac{1}{10^6 \cdot t^{1/4}}$.
Then by the elementary inequality $|e^x - 1| \leq 2|x|$ for all $x \in [-1, 1]$, and so
$$|e^{4s^3/t} - 1|, |e^{-4s^3/t}-1| \leq \frac{1}{10 \cdot t^{1/4}}.$$

Putting these together, we get that $\left| \tilde{P}_{i,\lambda_\ell} - \left(\frac{\ell}{s+2}\right)^i  \right| \leq \frac{1}{10 \cdot t^{1/4}}$
for each $\ell$.

Putting this back into \eqref{eqerr}, we get inequality~\eqref{lemeqerr} .
\end{proof}

\begin{lemma}
	\label{lem-quad}
	Let $c_0, \ldots, c_s \in \mathbb R$.

	There exists $ \ell \in \{1, 2, \ldots, s+1\}$ s.t.
	$$ \left| \sum_{i=0}^{s} c_i \left(\frac{\ell}{s+2} \right)^i \right| \geq \frac{1}{(10s)^{10s^2}} \cdot \left(\sum_{i=0}^s |c_i|^2 \right)^{1/2}.$$
\end{lemma}
\begin{proof}
	Let $Q : \mathbb R^{s+1} \to \mathbb R$ denote the quadratic form:
	$$ Q(c_0, \ldots, c_s) = \sum_{\ell = 1}^{s+1} \left( \sum_{i=0}^s c_i \left(\frac{\ell}{s+2}\right)^i \right)^2.$$
	We also use $Q$ to denote the matrix associated with this quadratic form.

	Note that $Q$ is positive definite (positive semi-definiteness is clear; to get positive definiteness, one needs to use
	the fact that a nonzero polynomial of degree at most $s$ cannot vanish at $s+1$ points).

	We now show that the smallest eigenvalue of $Q$ is at least $\frac{1}{ (10s)^{20s^2} }$.
	Using the Cauchy-Schwarz inequality, it is easy to see that $Q(c_0, \ldots, c_s) \leq (s+1)^2 \cdot (\sum_i c_i^2 )$, and thus
the top eigenvalue $\lambda_1$ of $Q$ is at most $(s+1)^2$. Furthermore, the determinant of $Q$ is a nonzero rational number
with denominator at most $(s+2)^{2s(s+1)}$. Thus the determinant of $Q$ is at least $\frac{1}{(s+2)^{2s(s+1)}}$.
Since the product of the eigenvalues equals the determinant, we conclude that the smallest eigenvalue of $Q$
is at least $\frac{\det(Q)}{\lambda_1^{s-1}} \geq \frac{1}{(s+2)^{2s ( s+1) + 2(s-1)}} \geq \frac{1}{(10s)^{10 s^2}}$.

	If the conclusion of the lemma does not hold, then
	$$Q(c_0, \ldots, c_s) \leq \frac{s}{(10s)^{20s^2}} \left(\sum_{i=0}^s |c_i|^2\right).$$ 
	
	This contradicts the above bound on the smallest eigenvalue of $Q$.
\end{proof}

By the Cauchy-Schwarz inequality, 
we have $(\sum_{i=0}^s |b_i|^2)^{1/2} \geq \frac{C}{\sqrt{s}}$.

By Lemma~\ref{lem-quad}, there exists $\ell \in \{1,2, \ldots, s+1\}$ such that 
$$|R(\frac{\ell}{s+2} )| \geq \frac{1}{(10s)^{10s^2}} \cdot \frac{C}{\sqrt{s}}.$$

Combining this with Lemma~\ref{lemerr}, we get:
\begin{align*}
 | \langle b, y_\ell \rangle | &\geq \frac{1}{(10s)^{10s^2}} \cdot \frac{C}{\sqrt{s}} - \frac{C}{t^{1/4}}\\
&\geq C \cdot \left( \frac{1}{(10s)^{11s^2}} - \frac{1}{t^{1/4}}\right).
\end{align*}

Since $s \leq \frac{1}{100} (\log t)^{1/5}$, we have that $(10s)^{11 s^2} < \frac{1}{2} t^{1/4}$, and so
$$\left( \frac{1}{(10s)^{11s^2}} - \frac{1}{t^{1/4}}\right) \geq \frac{1}{(20 s)^{20s^2}}.$$
Thus 
$$ | \langle b, y_\ell \rangle | \geq \frac{C}{(20s)^{20s^2}}.$$
But we know that $| \langle b, y_\ell \rangle| \leq 1$.

This implies $C \leq (20s)^{20s^2}$, as desired.
\end{proof}

We now complete the proof of Theorem~\ref{mainthm}.
By Equation~\eqref{eqbp}, 
$$ p_s = \frac{-1}{b_s} \cdot \left(\sum_{i=0}^{s-1} p_i b_i \right).$$
By Lemma~\ref{bbound}, we have that $|b_i| \leq (20s)^{20 s^2}$ for each
$i \leq s-1$.
Since $b_s \neq 0$, we have $|b_s| \geq 1$.
Thus:
$$|p_s| \leq (20s)^{20s^2} \cdot (\sum_{i = 0}^{s-1} |p_i| ).$$
On the other hand, the hypothesis tells us that $|p_s| > f(s) \cdot (\sum_{i = 0}^{s-1} |p_i| )$ (since $s \geq 1$).
But $(20s)^{20s^2} < f(s)$; this gives the desired contradiction.

This completes the proof of the theorem.

\end{proof}

Note that our main result also implies that for every torsion free group $G$ and any infinite cardinal $\aleph$, there is a $G$-valued SNZ charge on the Boolean algebra of clopen sets of $\{0,1\}^{\aleph}$.

Proposition 14 of~\cite{GB} shows that for every constant $c$, we may not
take $f(k) = c^k$ in the statement of Theorem~\ref{mainthm}.
It would be interesting to know how small we may take $f(k)$ in this theorem.

\section{$\Z$-valued SNZ charges on ${\mathcal P}(N)$}
\label{sectionPN}

We shall now consider the problem of existence of $\Z$-valued SNZ charges on ${\mathcal P}(\mathbb N)$. Proposition 12 of~\cite{GB} implies that there is a $\Z^{\mathcal{P}(N)}$-valued SNZ charge on $\mathcal{P}(\mathbb N)$. Below we show that
there is no $\Z$-valued SNZ charge on $\mathcal{P}(\N)$.

\begin{theorem}
\label{thmpn}
There is no $\Z$-valued SNZ charge on ${\mathcal P}(N)$. There is no $\mathbb Q$ valued SNZ charge on ${\mathcal P}(N)$.
\end{theorem}
\begin {proof}
Let us first see that in  ${\mathcal P}(N)$ there is a chain of cardinality of the continuum $c$. This is a folklore result. We give a simple argument for completeness.  Enumerate the rationals in $R$ as $q_1, q_2, \cdots ...$.For every real number $r$, let $A_r$ be the set $\{i: q_i <  r\}$. Then $A_r : r \mbox{  is a real number} \}$ is a chain of distinct sets of the cardinality of the contiunuum $c$. 

Of course, if there is a $\Z$-valued SNZ charge $\mu$ on ${\mathcal P}(N)$, $\mu(A_r) \neq \mu(A_s)$ if $r \neq s$. But $\mu$ can only take countably many 
values since $\mu$ is $\Z$-valued. Hence there is no $\Z$-valued SNZ charge on ${\mathcal P}(N)$. The same argument works for $\mathbb Q$-valued SNZ charges too.
\end{proof}

This raises an interesting problem. If $\mu$ is a $\Z$-valued SNZ charge  on a Boolean algebra $\mathcal A$ and if $\mathcal B \supset \mathcal A$ is another Boolean algebra, then under what conditions does there exist an extension of $\mu$   to a  $\Z$-valued SNZ charge on $\mathcal B$? In the next theorem we shall see some necessary conditions.

\begin{theorem}
\label{thm:mod}
Let $\mu $ be a $\Z$ valued  SNZ charge on a Boolean algebra $\mathcal A$. Suppose that  $\{A_i : i \in \N \}$  is an infinite family of  pairwise disjoint nonempty sets  in $\mathcal A$ such that $\mu(A_i) = a$ for all $i$.
then  there is a Boolean algebra $\mathcal B \supset \mathcal A$ such that $\mu$ cannot be extended as a $\Z$-valued SNZ charge on the Boolean algebra  $\mathcal B$.
\end{theorem}

\begin {proof}
Since $\mu$ is SNZ, $a$ is a nonzero integer. Take a strictly decreasing sequence of subsets $D_1, D_2, \cdots D_{a+1}$ of $\N$ so that $D_i - D_{i+1}$ is infinite for all $i$. For $1 \leq k \leq a+1$, let $E_k = \cup_{i \in D_k} A_i$. Then $E_1, E_2, \cdots E_{a+1}$ is a strictly increasing sequence of sets and  $\{\mu(E_1) , \mu(E_2), \cdots , \mu(E_{a+1})\} $ is a set of $a+1$ many distinct integers. Hence there exist integers $\ell$ and $m$, with $\ell < m$, such that $a$ divides $\mu(E_\ell) - \mu(E_m)$. Hence $E_m - E_l$, call it $F$, is a nonempty set of $\mu$ measure equal to $pa$. But this set also is an infinite union of $A_i'$s. Hence, if we take a union of $p$ many of these $A_i'$s and call it $G$, then $\mu(G) = pa$. Hence $G \subset F$, $\mu(G) = \mu(F)$ and $G \neq F$. Hence $\mu(F - G) = 0$. Thus $\mu$ cannot be extended as a $\Z$ valued SNZ charge on the Boolean algebra generated by $\mathcal A$ and  $E_1, E_2, \cdots E_{a+1}$.
\end {proof}

Let us consider the $\Z$-valued SNZ charge $\mu$  on the finite cofinite Boolean algebra $\mathcal A$ on $N$ defined by $\mu(A) = \#(A)$ if $A$ is finite and $= -1 - \#(A^c)$ if $A$ is cofinite. By Theorem~\ref{thmpn} this charge cannot be extended to ${\mathcal P}(N)$ as a $\Z$-valued SNZ charge. By the proof of  Theorem~\ref{thm:mod}  there is a Boolean algebra $\mathcal B$ which is generated by $\mathcal A$ and finitely many sets so that $\mu$ cannot be extended as a $\Z$ valued SNZ charge. 

In fact more is true for this charge. $\mu$ cannot be extended as a $\Z$-valued SNZ charge on the Boolean algebra generated by $\mathcal A$ and the set $E$ of even numbers. The proof is left as an exercise. This gives an constructive negative answer to the following question: If $\mu$ is a $\Z$-valued SNZ charge on a Boolean algebra ${\mathcal A} $ and if ${\mathcal B} $ is the Boolean algebra generated by $\mathcal A$ and a set $C$, should there exist an SNZ extension of $\mu $ to $\mathcal B$? A nonconstructive negative answer to this question can be deduced from Theorem~\ref{thmpn} and Zorn's lemma.

\section{Problems}

The problem of finding a combinatorial necessary and sufficient condition for the existence of a $\Z$-valued SNZ charge seems to be quite interesting.

Let ccc denote the countable chain condition: every collection of pairwise disjoint sets is countable. Let ecc denote the condition: every chain is countable.
By a result of~\cite{GB}, every nonatomic Boolean algebra which admits
an SNZ $\Z$-valued charge satisfies ccc and ecc.  
If $\B$ is a nonatomic Boolean algebra that satisfies both ccc and ecc
then should  $\B$ admit an SNZ $\Z$-valued charge? We suspect not.

\appendix
\section{Proof of Lemma~\ref{lemdual}}

\begin{proof}
By definition,
$$\langle u_i, v_k \rangle = \sum_{j=0}^t (-1)^{k-j} {t-i \choose t-j} {t-j \choose t - k}.$$

The $i = k$ case follows by observing that the only nonzero term in the above sum comes from $j = i = k$.

Now we deal with the case $i \neq k$.
Let $A(X)$ be the polynomial given by:
$$A(X) = (1-X)^{t-i} = \sum_{j=i}^{t} (-1)^{t-j} {t-i \choose t-j} X^{t-j} = \sum_{j=0}^{t} (-1)^{t-j} {t-i \choose t-j} X^{t-j}.$$
Note that the $p$'th derivative $A^{(p)}(1)$ equals zero in the following two cases:
\begin{itemize}
\item $p < t-i$: Then $A^{(p)}(X)$ is divisible by $(1-X)$, and so $A^{(p)}(1) = 0$.
\item $p > t-i$: Then $A^{(p)}(X)$ is the $0$ polynomial, since $A$ has degree $t-i$. In particular,
$A^{(p)}(1) = 0$.
\end{itemize}

Finally, by differentiating term-by-term, we see that
$$\frac{1}{p!} A^{(p)}(X) = \sum_{j=0}^t (-1)^{t-j} {t-i \choose t-j} {t-j \choose p} X^{t-j-p}.$$
Substituting $p = t-k$, $X = 1$, and using the above observations on $A^{(p)}(1)$,
the lemma follows.
\end{proof}

s\end{document}